\documentclass[12pt,leqno,fleqn]{amsart}  
\usepackage{tikz}
\usepackage{amsmath,amstext,amsthm,amssymb,amsxtra}
\usepackage[top=1.5in, bottom=1.5in, left=1.25in, right=1.25in]	{geometry}
\usepackage{lmodern}

\usepackage{mathtools}
\mathtoolsset{showonlyrefs,showmanualtags}

\usepackage{hyperref} 
\hypersetup{
	colorlinks=true,       
	linkcolor=blue,          
	citecolor=magenta,        
	filecolor=magenta,      
	urlcolor=cyan           
}

\usepackage[msc-links]{amsrefs}

\newtheorem{theorem}{Theorem}[section]

\newtheorem{lemma}{Lemma}[section]
\newtheorem{prob}{Question}
\newtheorem{corollary}{Corollary}[section]
\newtheorem{OldTheorem}{Theorem}
\theoremstyle{definition}

\theoremstyle{definition}

\theoremstyle{remark}
\newtheorem{remark}{Remark}[section]

\numberwithin{equation}{section}

\def\t{\ensuremath{\textbf t}}

\def\ZM{\ensuremath{\mathfrak M}}

\def\ZA{\ensuremath{\mathcal A}}

\def\ZZ{\ensuremath{\mathbb Z}}

\def\ZN{\ensuremath{\mathbb N}}

\def\ZR{\ensuremath{\mathbb R}}

\def\E{\ensuremath{\textbf E}}

\def\md#1#2\emd{\ifx0#1
	\begin{equation*} #2 \end{equation*}\fi  
	\ifx1#1\begin{equation}#2\end{equation}\fi   
	\ifx2#1\begin{align*}#2\end{align*}\fi   
	\ifx3#1\begin{align}#2\end{align}\fi    
	\ifx4#1\begin{gather*}#2\end{gather*}\fi  
	\ifx5#1\begin{gather}#2\end{gather}\fi   
	\ifx6#1\begin{multline*}#2\end{multline*}\fi  
	\ifx7#1\begin{multline}#2\end{multline}\fi  
	\ifx8#1\begin{multline*}\begin{split}#2\end{split}\end{multline*}\fi
	\ifx9#1\begin{multline}\begin{split}#2\end{split}\end{multline}\fi
}
\newcommand {\e }[1]{\eqref{#1}}
\newcommand {\lem }[1]{Lemma \ref{#1}}

\newcommand {\cor }[1]{Corollary \ref{#1}}

\newcommand {\trm }[1]{Theorem \ref{#1}}

\title[] {Sharp inequalities involving multiplicative chaos sums}

\author{Grigori A. Karagulyan}
\address{Faculty of Mathematics and Mechanics, Yerevan State
	University, Alex Manoogian, 1, 0025, Yerevan, Armenia} 
\email{g.karagulyan@ysu.am}
\address{Institute of Mathematics of NAS of RA, Marshal Baghramian ave., 24/5, Yerevan, 0019, Armenia} 
\email{g.karagulyan@gmail.com}

\thanks{The work was supported by the Science Committee of RA, in the frames of the research project 21AG‐1A045 }

	\begin{document}
\begin{abstract}
The present note is an essential addition to author's recent paper \cite{Kar4}, concerning general multiplicative systems of random variables. Using some lemmas and the methodology of \cite{Kar4}, we obtain a general extreme inequality, with corollaries involving Rademacher chaos sums and those analogues for multiplicative systems. In particular we prove that a system of functions generated by bounded products of a multiplicative system is a convergence system. 
\end{abstract}

	\maketitle  
\section{Introduction}
	A sequence of bounded random variables  $\phi_n$, $n=1,2,\ldots$ (finite or infinite) is said to be multiplicative if the equality
	\begin{equation}\label{a13}
	{\textbf E}\left[\prod_{j\in A}\phi_{j}\right]=0
	\end{equation}
	holds for all nonempty subsets $A\subset \ZN$ of positive integers.  	The multiplicative systems were introduced by Alexits (see \cite {Alex}, chap 3.2), who also proved in \cite{Alex1} that the uniformly bounded multiplicative systems are convergence systems, i.e. the condition $\sum_k a_k^2<\infty$ implies almost sure convergence of series  $\sum_k a_k\phi_k$ (see \cite{KoRe, Kom, Gap2} for further extensions of this result).  The multiplicative systems were considered also in the context of central limit theorems and law of iterated logarithm (\cite{Mor, Ber, Fuk1,Fuk2}). Well-known examples of multiplicative sequences are mean zero independent random variables and more general, the martingale-differences. The trigonometric sequences $\{\cos2n_k\pi (x+x_k)\}$ on $(0,1)$ with positive integers $n_k$, satisfying $n_{k+1}\ge 2n_k$, are known to be non-martingale examples of multiplicative systems (see \cite{Zyg1}, chap. 5).  A wide class of multiplicative systems was recently introduced in \cite{Rub}. Namely, under certain symmetry conditions on $f\in L^\infty(\ZR/\ZZ)$, the functions $f(nx)$, $n=1,2,\ldots$ form a multiplicative system on $(0,1)$.

		Recall the definition of multiplicative error given in \cite{Kar4}. Denote $\ZZ_n=\{1,2,\ldots,n\}$ and let $\ZZ_\infty=\{1,2,\ldots\}$ be the set of all positive integers. We consider systems of random variables $\phi=\{\phi_k:\, k\in \ZZ_n\}$ (here $n$ can be either finite or infinite) satisfying
	\begin{equation}\label{a28}
		A_k\le \phi_k\le B_k,  \text{ where }A_k<0<B_k.
	\end{equation}
Given $d\le n\le \infty$ the $d$-multiplicative error of the system $\phi$ is the quantity 
	\begin{equation}
	\mu_d(\phi)=\sum_{A\subset  \ZZ_n,\, \#(A)\le d}\quad\prod_{j\in A}(C_j)^{-1}\left|{\textbf E}\left[\prod_{j\in A}\phi_{j} \right]\right|,
	\end{equation}
where $C_k=\min\{-A_k,B_k\}$.  We say $\phi$ is $d$-multiplicative if $\mu_d(\phi)=0$ or equivalently \e{a13} holds for any finite nonempty $A\subset \ZN$ of cardinality $\#(A)\le d$. If $d=\infty$ this gives the classical definition of multiplicative systems. It was proved in \cite{Kar4}  (see \lem{L0} below)  that if the system $\phi$ is defined on $[0,1]$ and $\mu=\mu_d(\phi)<\infty$, then the functions of $\phi$ can be extended up to $[0,1+\mu]$ to a $d$-multiplicative system, obeying same condition \e{a28}. We call a system of random variables $d$-independent if any collection of its $d$ elements is independent. It is clear that $d$-independent mean zero random variables are $d$-multiplicative. 

The main result of \cite{Kar4} reads as follows.
	\begin{OldTheorem}[\cite{Kar4}]
		Let $\Phi:{\mathbb R}\to {\mathbb R}^+$ be a convex function and $\phi=\{\phi_k:\, k=1,2,\ldots, n\}$ be a system  of random variables satisfying \eqref{a28}.  Then for any integer $d\le n$ and a choice of real coefficients $a_1,\ldots, a_n$ it holds the inequality
		\begin{equation}\label{x12}
			\textbf{E}\left[\Phi\left(\sum_{k=1}^na_k\phi_k\right)\right]\le 	(1+\mu_d(\phi))\textbf{E}\left[\Phi\left(\sum_{k=1}^na_k\xi_k\right)\right],
		\end{equation}
		where $\xi_k$, $k=1,2,\ldots,n$, are $\{A_k,B_k\}$- valued $d$-independent mean zero random variables. 
	\end{OldTheorem}
Applying this inequality for $d$-multiplicative systems, i.e. $\mu_d(\phi)=0$, with the parameters $A_k=-1$ and $B_k=1$, we immediately obtain the following extremal property of Rademacher random variables.
	\begin{OldTheorem}[\cite{Kar4}]
		Let $\Phi:{\mathbb R}\to {\mathbb R}^+$ be a convex function. If $\phi=\{\phi_k:\, k=1,2,\ldots, n\}$, is a multiplicative system of random variables, satisfying $\|\phi_k\|_\infty\le 1$, then for any choice of coefficients $a_1,\ldots, a_n$ we have
		\begin{equation}\label{a50}
			\textbf{E}\left[\Phi\left(\sum_{k=1}^na_k\phi_k\right)\right]\le {\textbf E}\left[\Phi\left(\sum_{k=1}^na_kr_k\right)\right],
		\end{equation}
		where $r_k$, $k=1,2,\ldots,n$, are Rademacher independent random variables.
	\end{OldTheorem}
Applying inequalities \e{x12} and \e{a50}, paper \cite{Kar4} provides extensions of some probability inequalities for general multiplicative systems. In particular it was proved the following generalization of Azuma-Hoeffding's martingale inequality \cite{Azu, Hoe}. 
\begin{OldTheorem}[\cite{Kar4}]
	If a system  of random variables $\phi=\{\phi_k:\, k=1,2,\ldots, n\}$ satisfies \eqref{a28}, then it holds the inequality
	\begin{equation}\label{a21}	
		\left|\left\{\sum_{k=1}^n\phi_k>\lambda\right\}\right|\le (1+\mu_n(\phi)\exp\left(-\frac{2\lambda^2}{\sum_{k=1}^n(B_k-A_k)^2}\right),\quad \lambda>0.
	\end{equation}
\end{OldTheorem}
The present note is an addition to paper \cite{Kar4}. Essentially using the lemmas and the methodology of \cite{Kar4} it provides an extension of inequality \e{x12} in more general settings, with corollary inequalities involving Rademacher chaos and its analogue for multiplicative systems. We say a function $G:{\mathbb R}^n\to {\mathbb R}^+$ is convex, if it is convex with respect to each variable separately. Hence the main result of the present note is the following.
	\begin{theorem}\label{T1}
		Let a system  of random variables $\phi=\{\phi_k:\, k=1,2,\ldots, n\}$ satisfy \eqref{a28}, $d\le n$ and $G:{\mathbb R}^n\to {\mathbb R}^+$ be a convex function.  Then it holds the inequality
		\begin{equation}\label{a12}
		\textbf{E}\left[G\left(\phi_1,\ldots,\phi_n\right)\right]\le (1+\mu_d(\phi))\textbf{E}\left[G\left(\xi_1,\ldots,\xi_n\right)\right],
		\end{equation}
		where $\xi_k$, $k=1,2,\ldots,n$, are $\{A_k,B_k\}$- valued $d$-independent mean zero random variables.
	\end{theorem}

	\begin{corollary}\label{C1}
		Let $G:{\mathbb R}^n\to {\mathbb R}^+$ be a convex function. If $\phi=\{\phi_k:\, k=1,2,\ldots, n\}$, is a system of random variables, satisfying $\|\phi_k\|_\infty\le 1$, (i.e. $B_k=-A_k=1$ in \e{a28}), then we have
		\begin{equation}\label{a33}
		\textbf{E}\left[G\left(\phi_1,\ldots,\phi_n\right)\right]\le (1+\mu_d(\phi))\textbf{E}\left[G\left(r_1,\ldots,r_n\right)\right],
		\end{equation}
		where $r_k$, $k=1,2,\ldots,n$, are Rademacher $d$-independent random variables.
	\end{corollary}
\begin{remark}
	We remark that if $l<n$, then the sequences of $\{A_k,B_k\}$- valued mean zero $d$-independent random variables can be probabilistically different, whereas when $d=n$ such a sequence is uniquely determined. So if $d=n$ in \cor{C1},  then $r_k$, $k=1,2,\ldots,n$, become ordinary Rademacher independent random variables. In fact, the case $d=n$ is the most interesting case in \trm{T1} and \cor{C1}.
\end{remark}
\begin{remark}
	We will mainly be interested in convex functions $G$ obtained as follows. Let $X$ be a Banach space and $\ZA\subset 2^{\ZZ_n}$ be a family of nonempty subsets of $\ZZ_n=\{1,2,\ldots,n\}$. To each $A\in\ZA$ we assign a coefficient $x_A\in F$ and consider the function
\begin{equation}
	P(\t)=\sum_{A\in \ZA}x_A\prod_{k\in A}t_k:\ZR^n\to X,
\end{equation}
with respect to variable $\t=(t_1,\ldots,t_n)\in \ZR^n$. We will call such functions quasi-polynomials. Given sequence of quasi polynomials $P_j$, $j=1,2,\ldots,m$, and a one variable convex function $\Phi:\ZR^+\to \ZR^+$. It is an easy task to check that
\begin{equation}
	G(t_1,\ldots,t_n)=\Phi\left(\max_{1\le j\le m}\|P_j(t_1,\ldots,t_n)\|_X\right)
\end{equation}
is a convex function on $\ZR^n$.
\end{remark}
Let $\phi_n$, $n=1,2,\ldots$, be a multiplicative system. For a given nonempty $A\subset \ZN$ denote
\begin{equation}\label{x13}
	\phi_A(x)=\prod_{k\in A}\phi_k(x).
\end{equation}
If $\{\phi_n\}$ coincides with the Rademacher system, $\phi_n=r_n$, then $\{\phi_A\}$ becomes the classical Walsh system that will be denoted by $w_A(x)$. The subsystem $\{w_A:\, A\subset \ZN,\, \#(A)=d\}$ of Walsh functions, where only $d$-products of Rademacher functions are involved, is called Rademacher chaos of order $d$. The Rademacher chaos of first order ($d=1$) coincides with the Rademacher functions. The Rademacher chaos systems are very well investigated in different points of view. It is of interest the comparisons of the different $L^p$-norms of finite sums
\begin{equation}\label{d1}
	S(x)=\sum_{A\subset \ZN,\, \#(A)=d}b_Aw_A(x),
\end{equation}
 where $b_A$ is finitely supported sequence of real numbers. There are many other interesting properties of Rademacher chaos systems that doesn't have the complete system of Walsh functions (see \cite{Ast}, chap. 6). The following generalization of Khintchine \cite{Khi} inequality for the Rademacher chaos proved independently by Bonami \cite{Bon} and Kiener \cite{Kie} is a crucial tool for the analysis of Boolean functions (see \cite{Don}).
\begin{OldTheorem}[Bonami-Kiener]
	For any integers $d\ge 2$ and a Rademacher chaos sum \e{d1} it holds the bound
	\begin{equation}\label{d2}
		\left\|S\right\|_p\le (p-1)^{d/2}\left\|S\right\|_2,\quad p>2.
	\end{equation}
\end{OldTheorem} 
It was also proved in \cite{Bon} that the constant growth in \e{d2} as $p\to \infty$ is optimal in the sense that  for any $p>2$ there exists a sum \e{d1} such that 
$	\|S\|_p\ge C_dp^{d/2}\|S\|_2$ with a constant $C_d>0$, depending only on $d$. Using a standard argument one can deduce from inequalities \e{d2} the bound
\begin{equation}\label{d3}
	\left\|S\right\|_p\le C_{p,q,d}\left\|S\right\|_q,
\end{equation} 
where $1\le q<p<\infty$. It is of interest to find the optimal constants in \e{d3}. The case of the classical Rademacher functions (when $d=1$) is well investigated. The analogue of inequality \e{d2} for the simple Rademacher sums is Khintchine's classical inequality and there were found the values of all sharp constants in \e{d3} if either $p$ or $q$ is equal $2$ (see \cite{Ste, Sza, Haa, You}). To the best of our knowledge the sharp values of these constants if $d\ge 2$ are not known for any combination of the parameters $p> q$. Some particular estimates of constants \e{d3} one can find in papers \cite{Jan, Don, Lar, Ivan}. The following inequalities are immediately obtained from \e{a12}.
\begin{corollary}\label{C2}
	Let $\phi_n$, $n=1,2,\ldots$, be a system of random variables with $\|\phi_n\|_\infty\le 1$, $2\le d\le \infty$, and $\Phi:\ZR_+\to \ZR_+$ be a convex function. If $A_n$ is a sequence of different nonempty finite subsets of  $\ZN$ with $\#(A_n)\le d$, then for any finitely supported sequence of vectors $b_n$ from a Banach space $X$ there hold the bounds
	\begin{align}
		&\E\left[\Phi\left(\left\|\sum_kb_{k}\phi_{A_k}\right\|_X\right)\right]\le(1+\mu_d(\phi)) \E\left[\Phi\left(\left\|\sum_kb_{k}w_{A_k}\right\|_X\right)\right],\label{x14}\\
		&\E\left[\Phi\left(\max_n\left\|\sum_kb_{k}\phi_{A_k}\right\|_X\right)\right]\le(1+\mu_d(\phi)) \E\left[\Phi\left(\max_n\left\|\sum_kb_kw_{A_k}\right\|_X\right)\right].\label{x17}
	\end{align}
\end{corollary}

Letting $\Phi(t)=t^p$ in \e{x14} and $X$ to be the Banach space of real line $\ZR$, we will get inequality \e{d2} for more general sums
\begin{equation}\label{x16}
		S(x)=\sum_{A\subset \ZN,\, \#(A)=d}b_A\phi_A(x),
\end{equation}
where $\phi_A$ are the product random variables \e{x13} generated by a multiplicative systems. 
\begin{theorem}
	$\{\phi_n:\, n=1,2,\ldots\}$ with $\|\phi_n\|_\infty\le 1$, then for any integers $d\ge 2$ and any sum \e{x16} it holds the bound
	\begin{equation}\label{d10}
		\left\|S\right\|_p\le (p-1)^{d/2}\left\|S\right\|_2,\quad p>2.
	\end{equation}
\end{theorem} 
Recall that a system of functions $\{\phi_n\}$ is said to be a convergence system if any series 
\begin{equation}\label{x15}
	\sum_na_n\phi_n(x)
\end{equation}
with coefficients satisfying $\sum_na_n^2<\infty$ converges a.e., if a.e. convergence holds after any rearrangement of the terms of \e{x15}, then we say $\phi_n(x)$ is an unconditional convergence system.  As it was proved by Stechkin \cite{Ste}, if a system of functions $\phi_n\in L^2$ satisfies a Khintchine type inequality
\begin{equation} 
	\left\|\sum_{k=1}^na_k\phi_k\right\|_p\le c\left(\sum_{k=1}^na_k^2\right)^{1/2}
\end{equation}
for some $p>2$ and $c>0$, then $\phi_n(x)$ is an unconditional convergence system. Using this result and the bound \e{d10} we conclude the following.
\begin{corollary}\label{C3}
	Let $\phi=\{\phi_n$, $n=1,2,\ldots\}$,  be an arbitrary system of random variables with $\|\phi_n\|_\infty\le M$ and $\mu_d(\phi)<\infty$, $d\ge 1$. Then for any sequence $A_n\subset \ZN$, $n=1,2,\ldots$, of different nonempty sets with $\#(A_n)\le d$ the random variables $\{\phi_{A_n}\}$ form an unconditional convergence system.
\end{corollary}
\begin{remark}
	An interesting case in \cor{C3} is when $\{\phi_k\}$ is a bounded martingale difference system, that is 
$
		\E\left(\phi_n|\phi_1,\ldots,\phi_{n-1}\right)=0,\quad n=2,3,\ldots.
$
\end{remark}
\begin{remark}
Inequalities \e{x14} and \e{x17} show certain extremal properties of the Rade\-ma\-cher chaos or the Walsh functions.  Some extremal properties of simple Rademacher functions one can find in \cite{Ast}, chap. 7 and 9. Those are mostly concerned to independent random variables and no one consider Rademacher chaos or maximal function of partial sums as we have in \e{x14} or \e{x17}.
\end{remark}

\section{Proof of \trm{T1}}\label{S3}
A real-valued  function $f$ defined on $[a,b)$ is said to be a step function if it can be written as a finite linear combination of indicator functions of intervals $[\alpha,\beta)\subset [a,b)$. The following lemmas were proved in \cite{Kar4}.
\begin{lemma}[\cite{Kar4}, Lemma 2.2]\label{L0}
	Let $\phi=\{\phi_k:\, k=1,2,\ldots,n\}$ be a system of measurable functions on $[0,1)$ satisfying \eqref{a28} and ${\mathfrak M}$ be an arbitrary family of subsets of ${\mathbb Z}_n$. Then the functions $\phi_k$ can be extended up to the interval $[0,1+\mu)$ such that
	\begin{align}
		&A_k\le \phi_k(x)\le B_k\text { if } x\in [0,1+\mu),\label{a26}\\
		&\int_0^{1+\mu}\prod_{j\in A}\phi_{j} =0\text { for all } A\in {\mathfrak M},\label{a27}
	\end{align}
	where 
		\begin{equation}
		\mu_\ZM(\phi)=\sum_{A\in \ZM}\quad\prod_{j\in A}(C_j)^{-1}\left|{\textbf E}\left[\prod_{j\in A}\phi_{j} \right]\right|,
	\end{equation}
	is the multiplicative error corresponding to the family $\ZM$.  Moreover, each $\phi_k$ is a step function on $[1,1+\mu)$.
\end{lemma}
\begin{lemma}[\cite{Kar4}, Lemma 2.3]\label{L1}
	Let measurable functions $\phi_k$, $k=1,2,\ldots,n$, defined on $[0,1)$, satisfy \eqref{a28}. Then for any $\delta>0$ one can find step functions $f_k$, $k=1,2,\ldots, n$, on  $[0,1)$ with $A_k\le f_k\le B_k$ such that
	\begin{equation}\label{a7}
		|\{|\phi_k-f_k|>\delta\}|<\delta,\quad k=1,2,\ldots,n,
	\end{equation}
	and 
	\begin{equation}\label{a31}
		\int_0^1f_{n_1}f_{n_2}\ldots f_{n_\nu}=0\text { as } \{n_1,n_2,\ldots,n_\nu\}\in {\mathfrak M},
	\end{equation}
	where ${\mathfrak M}$ is the family of finite nonempty sets $A\subset \ZN$, satisfying \eqref{a13}.
\end{lemma}
\begin{lemma}[\cite{Kar4}, Lemma 2.4]\label{L3}
	If $g_k$, $k=1,2,\ldots,n$, is a $d$-multiplicative system of nonzero random variables such that each $g_k$ takes two values, then $g_k$ are $d$-independent.
\end{lemma}
The following lemma is a version of Lemma 2.5 of \cite{Kar4} with the same proof. For the sake of completeness its complete proof will be provided. 
\begin{lemma}\label{L2}
	Let $f_1,f_2,\ldots,f_n$ be real-valued step functions on $[0,1)$ satisfying $A_k\le f_k(x)\le B_k$ and let $G:{\mathbb R}^n\to {\mathbb R}^+$ be a convex function. Then there are $\{A_k,B_k\}$-valued step functions $\xi_1,\xi_2,\ldots,\xi_n$ on $[0,1)$ such that 
	\begin{equation}\label{a8}
		\int_0^1\xi_{n_1} \xi_{n_2} \ldots \xi_{n_\nu}=\int_0^1f_{n_1}f_{n_2}\ldots  f_{n_\nu},
	\end{equation}
	for any choice of $\{n_1,n_2,\ldots,n_\nu\}\subset {\mathbb Z}_n=\{1,2,\ldots,n\}$, and it holds the inequality
	\begin{equation}\label{a24}
		\int_0^1G\left(f_1,\ldots,f_n\right)\le \int_0^1G\left(\xi_1,\ldots,\xi_n\right).
	\end{equation}
\end{lemma}
\begin{proof}
	Let $\Delta_j$, $j=1,2,\ldots,m$ be the constancy intervals of functions $f_k$. Let $\Delta=[\alpha,\beta)$ be one of those intervals. Observe that the point
	\begin{equation}
		c=\frac{B_1\alpha-A_1\beta}{B_1-A_1}+\frac{1}{B_1-A_1}\cdot \int_\alpha^\beta f_1(t)dt
	\end{equation}
	is in the closure of the interval $\Delta$. Then we define $\xi_1$ on $\Delta$ as
	\begin{equation}
		\xi_1(x)=B_1\cdot {\textbf 1}_{ [\alpha,c)}(x)+A_1\cdot{\textbf 1}_{[c,\beta)}(x),\quad x\in \Delta.
	\end{equation}
	Applying this to each $\Delta_j$, we will get a function $\xi_1$ defined on entire $[a,b)$ and one can check 
	\begin{equation}\label{a3}
		\int_{\Delta_j}\xi_1(t)dt=\int_{\Delta_j}f_1(t)dt,\quad j=1,2,\ldots,m.
	\end{equation}
	Since each $f_k$ is constant on the intervals $\Delta_j$, $j=1,2,\ldots,m$, from \eqref{a3} we conclude that
	\begin{equation}\label{a4}
		\int_0^1\xi_{1} f_{n_2} \ldots f_{n_\nu}=\int_0^1f_{1}f_{n_2}\ldots f_{n_\nu},
	\end{equation}
	for any collection of integers $1<n_2<\ldots<n_\nu\le n$. We may also claim that
	\begin{equation}\label{a5}
		\int_0^1G\left(f_1,f_2,\ldots,f_n\right)\le \int_0^1G\left(\xi_1,f_2,\ldots,f_n\right).
	\end{equation}
	Indeed, let us fix an interval $\Delta_j$ and suppose that $f_k(t)=c_k$ on $\Delta_j$. Applying \eqref{a3} and the Jessen inequality, we get
	\begin{align}
		\int_{{\Delta_j}}G\left(f_1(t),f_2(t),\ldots,f_n(t)\right)dt&=G\left(c_1,c_2,\ldots c_n\right)|{\Delta_j}|\\
		&=G\left(\frac{1}{|{\Delta_j}|}\int_{{\Delta_j}}\xi_1\,,c_2,\ldots,c_n\right) |{\Delta_j}|\\
		&\le \int_{{\Delta_j}}G\left(\xi_1(t),f_2(t),\ldots,f_n(t)\right)dt,
	\end{align}
	then the summation over $j$ implies \eqref{a5}. Applying the same procedure to the new system 
	$\xi_1,f_2,\ldots, f_n$ we can similarly replace $f_2$ by $\xi_2$. Continuing this procedure we will replace all functions $f_k$ to $\xi_k$ ensuring the conditions of lemma.
\end{proof}

\begin{proof}[Proof of Theorem \ref{T1}]
	At first we suppose that $\{\phi_k\}$ is a $d$-multiplicative system. Without loss of generality we can suppose that $\phi_k$ are defined on $[0,1)$. Applying Lemma \ref{L1}, we find an $d$-multiplicative system of step functions $f_k$ on $[0,1)$ satisfying 
	\begin{equation}\label{a32}
	|\{x\in [0,1):\,|f_k(x)-\phi_k(x)|>\delta\}|<\delta.
	\end{equation}
	Then we apply Lemma \ref{L2} and get $A_k,B_k$-valued step functions $\xi_k$ defined on $[0,1)$ and satisfying \eqref{a8} and \eqref{a24} ($a=0,b=1$). Since $\{f_k\}$ is $d$-multiplicative, in view of \eqref{a8} so we will have for $\{\xi_k\}$ and then by Lemma \ref{L3} $\{\xi_k\}$ forms $A_k,B_k$-valued $d$-independent random variables. Observe that
	\begin{equation}
	\int_0^1G\left(\phi_1,\ldots,\phi_n\right)\le \varepsilon(\delta) +\int_0^1G\left(f_1,\ldots,f_n\right),
	\end{equation}
where by \eqref{a32} we have
\begin{equation}\label{x9}
	\varepsilon =\varepsilon(\delta)\to 0\text{ as } \delta\to 0. 
\end{equation}
Here we also use the continuity of the convex function $G$. Thus, from \eqref{a24} we obtain
	\begin{align}
	\int_0^1G\left(\phi_1,\ldots,\phi_n\right)&\le \varepsilon +\int_0^1G\left(f_1,\ldots,f_n\right)\le \varepsilon +\int_0^1G\left(\xi_1,\ldots,\xi_n\right).
	\end{align}
	Therefore, taking into account \eqref{x9}, one can easily get \eqref{a12}. Indeed, the system $\{\xi_k\}$ in fact depends on $\delta$. So let $\{\xi_k^{(m)}\}$ be the system corresponding to $\delta =1/m$. We can suppose that there is a partition $0=x_0^{(m)}\le x_1^{(m)}\le \ldots\le x_{2^n}^{(m)}=1$ such that each function $\xi_k^{(m)}$ is equal to a constant (either $A_k$ or $B_k$) on each interval $[x_{j-1}^{(m)},x_j^{(m)})$, $1\le j\le 2^n$, and this constant is independent of $m$. Then we find a sequence $m_k$ such that each sequence $x_j^{(m_k)}$, $k=1,2,\ldots $, is convergent. One can check that this  generates a limit  sequence of $d$-multiplicative $\{A_k,B_k\}$-valued random variables $\xi_k$, $k=1,2,\ldots,n$, satisfying \eqref{a12}.
	
	Now suppose that $\{\phi_n\}$ is an arbitrary system. Applying \lem{L0} with $\ZM=\{A\subset \ZZ_n:\, \#(A)\le d\}$, we may extend  $\{\phi_n\}$ up to $[0,1+\mu]$, satisfying the conditions of the lemma.  Consider the $d$-multiplicative system $\psi_n(x)=g_n((1+\mu)x)$ on $[0,1) $. We have
	\begin{align}
\int_0^1G\left(\phi_1,\ldots,\phi_n\right)&\le \int_0^{1+\mu}G\left(\phi_1,\ldots,\phi_n\right)\\
&=(1+\mu)\int_0^1G\left(\psi_1,\ldots,\psi_n\right)\le (1+\mu)\int_0^1G\left(\xi_1,\ldots,\xi_n\right),
	\end{align}
where $\xi_n$ is a $d$-independent system generated by \lem{L3} applied on $\psi_n$. This completes the proof of theorem.
	
\end{proof}

\section{Trigonometric subsystems}

The Walsh functions have a canonical numeration, $w_n$, $n=0,1,2,\ldots$. That is $w_0\equiv 1$ and for $n\ge 1$ with a dyadic decomposition
\begin{equation}\label{x18}
	n=2^{k_1}+2^{k_2}+\cdots +2^{k_d},\quad 0\le k_1<k_2<\cdots<k_d.
\end{equation}
we have
\begin{equation*}
	w_n(x)=r_{k_1+1}(x)r_{k_2+1}(x)\ldots r_{k_d+1}(x).
\end{equation*}
We call $d$ a power of the number $n$ and write $\rho(n)=d$. 
Let $L_\Phi(0,1)$ denote the Orlicz space of functions on $(0,1)$, corresponding to a Young function $\Phi:\ZR_+\to\ZR_+$ that is a convex function satisfying
\begin{equation*}
	\lim_{t\rightarrow 0+}\frac{\Phi(t)}{t}=\lim_{t\rightarrow \infty }\frac{t}{\Phi(t)
	}=0.
\end{equation*}
The Luxemburg norm of a function $f\in L_\Phi(0,1)$ is defined by
\begin{equation*}
	\|f\|_\Phi=\inf \left\{ \lambda :\,\lambda
		>0,\,\int_0^1 \Phi\left( \frac{|f|}{\lambda }\right) \le
		1\right\} <\infty .
\end{equation*}

\begin{corollary}
If	$\Phi:\ZR_+\to \ZR_+$ is a Young function, then for any finitely supported sequence of coefficients $b_n$, $n=1,2,\ldots,$ we have the bound
	\begin{align}
&\left\|\sum_{n:\,\rho(n)\le d} b_n\cos 2\pi n(x+x_n)\right\|_\Phi\le 2^{d-1}\left\|\sum_{n:\,\rho(n)\le d} b_nw_n(x)\right\|_\Phi,\label{x19}\\
&\left\|\max_{m\ge 1}\left|\sum_{n\le m:\,\rho(n)\le d} b_n\cos 2\pi n(x+x_n)\right|\right\|_\Phi\le 2^{d-1}\left\|\max_{m\ge 1}\left|\sum_{n\le m:\,\rho(n)\le d} b_nw_n(x)\right|\right\|_\Phi\label{x20}
	\end{align}
\end{corollary}
\begin{proof}
	Using the trigonometric sums formulae and decomposition \e{x18}, the function $\cos 2\pi n(x+x_n)$ may be written as a sum of $2^{d-1}$ products of the form
	\begin{equation}
		t_1\big(2^{k_1}(x+\alpha_1)\big)t_2\big(2^{k_2}(x+\alpha_2)\big)\cdots t_s\big(2^{k_d}(x+\alpha_d)\big),
	\end{equation}
	where each $t_j(x)$ is either $\sin 2\pi x$ or $\cos 2\pi x$.  On the other hand for any choice of trigonometric functions $t_j(x)$ and numbers $\alpha_j$ the system $t_j(2^j(x+\alpha_j))$ is multiplicative and obeys the conditions of \cor{C2}. Thus, inequalities \e{x14} and \e{x17} imply \e{x19} and \e{x20} respectively.
\end{proof}
Then using the Bonami-Kiener inequality \e{d2}, from \e{x19} we immediately obtain the following 
\begin{corollary}
	For any finitely supported sequence $b_n$, $n=1,2,\ldots,$ we have the bound
	\begin{equation}\label{x21}
	\left\|\sum_{n:\,\rho(n)=d} b_nt_n(x)\right\|_p\le 2^{d-1}(p-1)^{d/2}\left\|\sum_{n:\,\rho(n)=d} b_nt_n(x)\right\|_2,\quad p>2.
\end{equation}
\end{corollary}
Note that inequality \e{x21} with another constant $c(d)$ instead of $2^{d-1}$ was proved in \cite{LoRo} (see Theorem 5.13). The method of proof in \cite{LoRo} is the Riesz products technique, quite a different from ours. An examination of its proof may provide a value $c(d)\ge  16^d$ for the mentioned constant. We have a question concerning the possible optimal constant in \e{x19}.
\begin{prob}
	Letting $\Phi(t)=t^p$, $1<p<\infty$, prove inequality \e{x19} with the right side constant independent of $d$.
\end{prob}
One can check that in the case
\begin{equation}
	\int_0^\infty \frac{dt}{\Phi(t)}=\infty,
\end{equation} 
the statement of this problem doesn't hold, since for the trigonometric and Walsh Dirichlet polynomials
\begin{equation}\label{x22}
	D_{2^n}^T(x)=\sum_{k=1}^{2^n} \cos 2\pi kx, \quad D_{2^n}^W(x)=\sum_{k=1}^{2^n} w_k(x)
\end{equation}
we have 
\begin{equation*}
\frac{\|D_{2^n}^T\|_\Phi}{\|D_{2^n}^W\|_\Phi}\to \infty \text { as } n\to \infty.
\end{equation*}
Polynomials \e{x20} tell nothing us about the optimal constant in \e{x20}. So in this case we suggest the following.
\begin{prob}
	Prove inequality \e{x20} with the right side constant independent of $d$ for any Young function $\Phi(t)$ .
\end{prob}

\bibliographystyle{plain}


\end{document}